\documentclass[12pt]{article}
\usepackage[centertags]{amsmath}
\usepackage{amsfonts}
\usepackage{amssymb}
\usepackage{latexsym}
\usepackage{amsthm}
\usepackage{newlfont}
\usepackage{graphicx}
\usepackage{listings}
\usepackage{booktabs}
\usepackage{abstract}
\usepackage{enumerate}
\usepackage{xcolor}
\RequirePackage{srcltx}
\date{}

\newlength{\defbaselineskip}
\newcommand{\setlinespacing}[1]%
           {\setlength{\baselineskip}{#1 \defbaselineskip}}

\newcommand{\mspan}{\operatorname{span}}
\newcommand{\N}{{\mathbb{N}}}

\newcommand{\actaqed}{\hfill $\actabox$}
{\medskip\noindent \textit{Proof of #1. }}%
{\actaqed \medskip}

\def\D{{\mathcal D}}

\def\cA{{\mathcal A}}

\def\cC{{\mathcal C}}
\def\cF{{\mathcal F}}
\def\FF{{\mathcal F}}

\def\cH{{\mathcal H}}
\def\CW{\mathcal{W}}
\def \Tr{\mathcal T}

\def \cN{\mathcal N}
\def \cS{\mathcal S}

\def \cX{\mathcal X}

\def\R{{\mathbb R}}
\def\Z{\mathbb Z}
\def\ZZ{\mathbb Z}

\def \T{\mathbb T}

\def \<{\langle}
\def\>{\rangle}

\def \Og{\Omega}

\def \e{\varepsilon}
\def \va{\varepsilon}

\def \ta{\theta}
\def \ff{\varphi}

\def\ga{\gamma}

\def \sp{\operatorname{span}}

\def\bx{\mathbf x}

\def\bk{\mathbf k}

\def\bs{\mathbf s}

\newtheorem{Theorem}{Theorem}[section]
\newtheorem{Lemma}{Lemma}[section]

\newtheorem{Remark}{Remark}[section]

\newtheorem{Corollary}{Corollary}[section]
\numberwithin{equation}{section}

\newcommand{\be}{\begin{equation}}
\newcommand{\ee}{\end{equation}}

\begin{document}

\title{Universal sampling discretization}

\author{ F. Dai and   V. Temlyakov 	\footnote{
		The first named author's research was partially supported by NSERC of Canada Discovery Grant
		RGPIN-2020-03909.
		The second named author's research was supported by the Russian Federation Government Grant No. 14.W03.31.0031.
  }}

\newcommand{\Addresses}{{
  \bigskip
  \footnotesize

  F.~Dai, \textsc{ Department of Mathematical and Statistical Sciences\\
University of Alberta\\ Edmonton, Alberta T6G 2G1, Canada\\
E-mail:} \texttt{fdai@ualberta.ca }

 \medskip
  V.N. Temlyakov, \textsc{University of South Carolina,\\ Steklov Institute of Mathematics,\\  Lomonosov Moscow State University,\\ and Moscow Center for Fundamental and Applied Mathematics.
  \\
E-mail:} \texttt{temlyak@math.sc.edu}

}}
\maketitle

\begin{abstract}
	{Let  $X_N$ be an $N$-dimensional subspace of $L_2$ functions on a probability space  $(\Og, \mu)$ spanned by  a uniformly bounded Riesz basis $\Phi_N$.  Given an integer   $1\leq v\leq N$ and an exponent $1\leq q\leq 2$, we obtain  universal discretization for
		integral norms $L_q(\Og,\mu)$ of functions from the collection  of  all  subspaces of $X_N$  spanned by $v$ elements of $\Phi_N$ with the number $m$  of required  points satisfying $m\ll  v(\log N)^2(\log v)^2$. 
		This last  bound on $m$   is much better than previously known bounds  which are  quadratic in $v$.   Our  proof uses  a  conditional theorem  on universal sampling discretization,  and an inequality of     entropy numbers  in terms of   greedy approximation  with respect to dictionaries.
	}
\end{abstract}


{\it MSC classification 2000:} Primary 65J05; Secondary 42A05, 65D30, 41A63.

%
%
%

\section{Introduction}
\label{I}

A standard approach to solving a continuous problem numerically -- the Galerkin method -- suggests to look for an approximate solution from a given finite-dimensional subspace. A typical way to measure an error of approximation is an appropriate $L_p$ norm, $1\le p\le\infty$. Thus, the problem of   discretization of the $L_p$ norms of functions from a given finite-dimensional subspace arises in a very natural way. Approximation by elements from a linear subspace falls in the category of {\it linear approximation}. 

It was understood in numerical analysis and approximation theory that in many problems from signal/image processing it is beneficial to use an $m$-term approximant with respect to a given system of elements (dictionary) $\D_N:=\{g_i\}_{i=1}^N$. This means that for $f\in X$ we look for an approximant of the form
\begin{equation}\label{P2}
a_m(f):=\sum_{k\in\Lambda(f)}c_kg_k
\end{equation}
where $\Lambda(f) \subset [1,N]$ is a set of $m$ indices which is determined by $f$. The complexity of this approximant is characterized by the cardinality $|\Lambda(f)|=m$ of $\Lambda(f)$. Approximation of this type is referred to as {\it nonlinear approximation} because,  for a fixed $m$, approximants $a_m(f)$ come from different linear subspaces spanned by $g_k$, $k\in \Lambda(f)$, which depend on $f$. The cardinality $|\Lambda(f)|$ is a fundamental characteristic of $a_m(f)$ called {\it sparsity} of $a_m(f)$ with respect to $\D_N$. It is now well understood that we need to study nonlinear sparse approximations in  order to significantly increase our ability to process (compress, denoise, etc.) large data sets.  Sparse approximations of a function  are not only a powerful analytic tool but they are utilized in many applications in image/signal processing and numerical computation. 

Therefore, here is an important ingredient of the discretization problem, desirable in practical applications. Suppose we have a finite dictionary $\D_N:=\{g_j\}_{j=1}^N$ of functions from $L_p(\Omega,\mu)$. Applying our strategy of sparse $m$-term approximation with respect to 
$\D_N$ we obtain a collection of all subspaces spanned by at most $m$ elements of $\D_N$ as a possible source of approximating (representing) elements. Thus, we would like to build a discretization scheme, which works well for all such subspaces. This kind of discretization falls in the category of {\it universal discretization}. The paper is devoted to the problem of universal sampling discretization.

Let  $\Omega$ be a nonempty set  equipped  with a  probability measure $\mu$.  For  $1\le p\leq  \infty$, let  $L_p(\Omega):=L_p(\Omega,\mu)$ denote  the real Lebesgue  space $L_p$
	defined with respect to the measure $\mu$ on $\Omega$, and  $\|\cdot\|_p$ the  norm of $L_p(\Og)$.
By discretization of the $L_p$ norm we understand a replacement of the measure $\mu$ by
a discrete measure $\mu_m$ with support on a set $\xi =\{\xi^j\}_{j=1}^m \subset \Omega$. This means that integration with respect to measure $\mu$ is replaced by an appropriate cubature formula. Thus, integration is replaced by evaluation of a function $f$ at a
finite set $\xi$ of points. This is why this way of discretization is called {\it sampling discretization}.
The problem of sampling discretization is a classical problem. The first results in this direction were obtained  in the 1930s by Bernstein, by Marcinkiewicz  and
by Marcinkiewicz-Zygmund  for discretization of the $L_p$ norms of the univariate trigonometric polynomials. Even though this problem is very important in applications, its systematic study has begun only  recently (see the survey paper \cite{DPTT}). We now give explicit formulations of the
sampling discretization problem (also known as the Marcinkiewicz discretization problem) and of the problem of universal discretization. 
   
{\bf The sampling discretization problem.} Let $(\Omega,\mu)$ be a probability space, and 
  $X_N\subset L_q$ an $N$-dimensional subspace of $L_q(\Omega,\mu)$ with  $1\le q \le \infty$  (the index $N$ here, usually, stands for the dimension of $X_N$).  We shall always assume that every function in $X_N$ is defined everywhere on $\Og$, and 
  \[ f\in X_N, \  \|f\|_q =0\implies f=0\in X_N.\]  We say that $X_N$  admits the Marcinkiewicz-type discretization theorem with parameters $m\in \N$ and $q$ and positive constants $C_1\le C_2$ if there exists a set $\xi := \{\xi^j\}_{j=1}^m \subset \Omega$
 such that for any $f\in X_N$ we have in the case $1\le q <\infty$
\be\label{A1}
C_1\|f\|_q^q \le \frac{1}{m} \sum_{j=1}^m |f(\xi^j)|^q \le C_2\|f\|_q^q
\ee
and in the case $q=\infty$ 
$$
C_1\|f\|_\infty \le \max_{1\le j\le m} |f(\xi^j)| \le  \|f\|_\infty.
$$

{\bf The  problem of universal discretization.} Let $\cX:= \{X(n)\}_{n=1}^k$ be a collection of finite-dimensional  linear subspaces $X(n)$ of the $L_q(\Omega)$, $1\le q \le \infty$. We say that a set $\xi:= \{\xi^j\}_{j=1}^m \subset \Omega $ provides {\it universal discretization} for the collection $\cX$ if, in the case $1\le q<\infty$, there are two positive constants $C_i$, $i=1,2$, such that for each $n\in\{1,\dots,k\}$ and any $f\in X(n)$ we have
\[
C_1\|f\|_q^q \le \frac{1}{m} \sum_{j=1}^m |f(\xi^j)|^q \le C_2\|f\|_q^q.
\]
In the case $q=\infty$  for each $n\in\{1,\dots,k\}$ and any $f\in X(n)$ we have
\be\label{1.2u}
C_1\|f\|_\infty \le \max_{1\le j\le m} |f(\xi^j)| \le  \|f\|_\infty.
\ee 

Note that the  problem of universal discretization for the collection $\cX:= \{X(n)\}_{n=1}^k$ is 
the sampling discretization problem for the set $\cup_{n=1}^k X(n)$. Also, we point out that the concept of universality is well known in approximation theory. For instance, the reader can find a discussion of universal cubature formulas in \cite{VTbookMA}, Section 6.8. 

The problem of universal discretization for some special subspaces of the
trigonometric polynomials was studied in \cite{VT160, DPTT}. To describe the results in \cite{VT160, DPTT}, we need to introduce some necessary notations.  First, given   a finite subset $Q$ of $\Z^d$, we set 
$$
\Tr(Q):= \Bigl\{f: f=\sum_{\bk\in Q}c_\bk e^{i(\bk,\bx)},\   \  c_{\mathbf{k}}\in \mathbb{C},\  \  \bk\in Q\Bigr\}.
$$
For $\bs=(s_1, \cdots, s_d)\in\Z^d_+$, we 
define
$$
R(\bs) := \{\bk \in \Z^d :   |k_j| < 2^{s_j}, \quad j=1,\dots,d\}.
$$
The following result, proved in \cite{VT160},  solves the universal discretization problem 
for    the collection 
$$
\cC(n,d):= \left\{\Tr(R(\bs)): s_1+\cdots+s_d=n\right\}
$$
 of subspaces of 
trigonometric polynomials. 
\begin{Theorem}\textnormal{\cite{VT160}}\label{udT1}   For every $1\le q\le\infty$ there exists a large enough constant $C(d,q)$, which depends only on $d$ and $q$, such that for any $n\in \N$ there is a set $\xi:=\{\xi^\nu\}_{\nu=1}^m\subset \T^d$, with $m\le C(d,q)2^n$ that provides universal discretization in $L_q$   for the collection $\cC(n,d)$.
\end{Theorem}

Second, for $n\in \N$, let
 $$\Pi_n :=[-2^{n-1}+1, 2^{n-1} -1]^d\cap \Z^d.$$
 For a positive integer $v\le |\Pi_n|$, define 
$$
\cS(v,n):= \left\{Q:  \  \  Q\subset \Pi_n, \   |Q|=v\right\}.
$$
Then it is easily  seen  that
\[
|\cS(v,n)| =\binom{|\Pi_n|}{v}<2^{dnv}.
\]
The following two theorems provide   universal discretization of $L_1$ and $L_2$ norms for the collection 
$\{\Tr(Q):\  \  Q\in \cS(v,n)\}$. 
 
\begin{Theorem}\textnormal{\cite[Theorem 7.4]{VT159,DPTT}}\label{udT4} For positive integers $n$ and $1\leq v\leq |\Pi_n|$, let 
	\[ M_p(n,v) :=\begin{cases}
	v^2 n^{9/2}, \   \ &\text{if $p=1$},\\
	v^2n,\   \ &\text{if $p=2$}.
	\end{cases}\]
	Then there exist three positive constants $C_i(d)$, $i=1,2,3$,
	such that for any $n,v\in\N$ with  $v\le |\Pi_n|$, and for  $p=1$ and $2$,  there is a set $\xi =\{\xi^\nu\}_{\nu=1}^m \subset \T^d$, with $m\le C_1(d) M_p(n,v)$ such that for any $f\in \cup_{Q\in \cS(v,n)} \Tr(Q)$
	$$
	C_2(d)\|f\|_p^p \le \frac{1}{m} \sum_{\nu=1}^m |f(\xi^\nu)|^p \le C_3(d)\|f\|_p^p.
	$$
\end{Theorem}

Let us denote by  $\D_N=\{g_i\}_{i=1}^N$  a system of functions from $L_p$. Denote the set of all $v$-term approximants with respect to $\D_N$ as
$$
\Sigma_v(\D_N):= \left\{f\,:\, f = \sum_{i\in G} c_ig_i,\quad\text{with any}\, G\subset [1,N],\quad |G|=v\right\}.
$$

Theorem \ref{udT4} provides universal discretization for the collection $\{\Tr(Q)\,:\, Q\in \cS(v,n)\}$, which is equivalent to the sampling discretization of the $L_p$ norm of elements from the set $\Sigma_v(\D_N)$ with $N=|\Pi_n|$, $\D_N =\{e^{i(\bk,\bx)}\}_{\bk\in \Pi_n}$. The proof of 
Theorem \ref{udT4} in the  case $p=2$ is based on deep results on random matrices and in the case 
$p=1$ is based on a chaining technique. We point out that in both the cases $p=2$ and $p=1$ Theorem \ref{udT4} provides universal discretization with the number of points growing as $v^2$. 
In this paper we prove the following estimate (see below for definitions and notations).
 
 \begin{Theorem}\label{IT3} Let $1\le p\le 2$. Assume that $\Phi_N$ is a uniformly bounded Riesz basis of $X_N:=\mspan(\Phi_N)$ satisfying \eqref{Riesz} for some constants $0<R_1\leq R_2$. Then for a large enough constant $C=C(p,R_1,R_2)$, and any integer $1\leq v\leq N$,  there exist 
$m$ points
$\xi^1,\cdots, \xi^m \in  \Omega$ with
$$
m \le Cv(\log N)^2(\log(2v))^2
$$
  such that for any $f\in  \Sigma_v(\Phi_N)$
we have
\begin{equation}\label{1.4u}
\frac{1}{2}\|f\|_p^p \le \frac{1}{m}\sum_{j=1}^m |f(\xi^j)|^p \le \frac{3}{2}\|f\|_p^p.
\end{equation}
\end{Theorem}
 
In particular, Theorem \ref{IT3} gives the order of bound 
$m\ll v(\log N)^2(\log (2v))^2$, which  is linear 
in $v$ with  extra logarithmic terms in $N$ and $v$.  This bound  is much better than previously known bounds (see Theorem \ref{udT4}), which provided quadratic in $v$ bounds. Note  that even for 
each individual subspace from $\Sigma_v(\Phi_N)$ we have the lower bound $m\ge v$ for the 
sampling discretization.

The rest of this paper is organized as follows.  Section \ref{X} and Section \ref{Y} are  devoted to estimating the  entropy numbers 
$\e_k(\Sigma_v^p(\Phi_N),L_\infty)$ 
of  the sets
$$
\Sigma_v^p(\Phi_N):= \{f\in \Sigma_v(\Phi_N)\,:\, \|f\|_p\le 1\},\  \ 1\leq p\leq 2,
$$
in the $L_\infty$-norm, where $\Phi_N$  is a  uniformly bounded Riesz basis   of
 $X_N:=[\Phi_N]\subset L_2$ and  $1\leq v\leq N$ is an integer. 
Such estimates  play an important role in the proof of Theorem \ref{IT3}.  To be more precise, in Section \ref{X}, we prove  under the  additional condition  \eqref{X5a} on the space $X_N=\mspan(\Phi_N)$ that for $p=2$, 
\be
\e_k(\Sigma_v^2(\Phi_N),L_\infty) \le C (\log N) \Bigl(\frac vk\Bigr)^{1/2},\quad   k=1,2,\dots.\label{1-4b}
\ee
The proof of \eqref{1-4b}  uses a  known result  from  Greedy approximation in smooth  Banach spaces and its connection with entropy numbers.  
In Section \ref{Y},  we show how the estimate \eqref{1-4b}  can be extended to the case $1\le p<2$ under the condition  \eqref{X5a}. 
  This  extension step  is based on a general inequality for 
 the entropy, which is  given in Lemma \ref{YL1} and  apprears to be of independent interest.
 In Section \ref{Z},  we prove Theorem \ref{IT3}, using the  estimates on entropy numbers established in the previous two sections, and  a conditional theorem on sampling discretization.  A main step  in the proof is to show that  the condition  \eqref{X5a} that   is assumed in our estimates of entropy numbers  can be dropped in sampling discretization.  The conditional Theorem \ref{XT2} used in the proof of Theorem \ref{IT3} is given in Section \ref{X} without proof.  
 In  Section \ref{sec:5},  we prove  a  refined   conditional theorem for  sampling  discretization of  all integral norms $L_q$ of functions from a    subset  $\CW\subset L_\infty$ satisfying certain conditions, which allows us to estimate  the number of  points  required for  the sampling discretization in terms of  an integral of the $\va$-entropy  $\mathcal{H}_\va (\CW, L_\infty)$, $\va>0$.   This  is an extension of the conditional result  proved in \cite{VT159,DPSTT2} for the unit ball of the space $X_N\subset L_p$.   In particular, it also  allows us to prove  a refined version of   Theorem \ref{IT3},  where  the constants $\frac 12$ and $\frac32$ in \eqref{1.4u} are  replaced by $1-\va$ and $1+\va$ respectively for an arbitrarily given $\va\in (0, 1)$.  Finally, in Section \ref{sec:6}, we give  a few  remarks on universal sampling discretization of $L_p$ norms for  $p>2$.


Throughout this paper, the letter $C$ denotes  a general positive
constant depending only on the parameters indicated as arguments or subscripts, and we will use the notation $|A|$ to denote the cardinality of a finite set $A$.

\section{Some general entropy bounds and the case $p=2$ }
\label{X}

It is well known that bounds of the entropy numbers of the unit ball of an $N$-dimensional subspace $X_N\subset L_p$ 
$$
X_N^p := \{f\in X_N\,:\, \|f\|_p \le 1\}
$$
play an  important role in sampling discretization of the $L_p$ norm of elements of $X_N$ (see \cite{VT158}, \cite{VT159}, \cite{DPSTT1}, and \cite{DPSTT2}). 

Recall the definition of entropy numbers in Banach spaces. 
Let $X$ be a Banach space and  $B_X(g,r)$ denote  the   closed ball $\{f\in X:\|f-g\|\le r\}$ with  center $g\in X$ and radius $r>0$.  Given a positive number $\e$,   the covering number $N_\e(A, X)$ of a compact set $A\subset X$ is defined  as
\[ N_\va (A,X):=\min\left\{ n\in\N:\  \ \exists\; g^1,\ldots, g^n\in A, \   A\subset \bigcup_{j=1}^n B_X(g^j, \va)\right\}.\]
    We denote by $\cN_\e(A,X)$
the corresponding minimal $\e$-net of the set $A$ in $X$; namely, $\cN_\va(A,X)$ is a finite subset of $A$ such that $A\subset \bigcup_{y\in \cN_\va(A,X)}B_X(y,\va)$ and  $N_\e(A,X)= |\cN_\e(A,X)|$.
The $\va$-entropy $\mathcal{H}_\va (A, X)$ of the compact set $A$  in $X$  is  defined as
$\log_2 N_\e(A,X)$, and    the entropy numbers $\e_k(A,X)$ of the set $A$ in $X$ are defined  as
\begin{align*}
\e_k(A,X)  :&=\inf \{\e>0: \cH_\e (A, X)\leq k\},\   \  k=1,2,\ldots.
\end{align*}

The  following  conditional result was proved in  \cite{VT159} for $p=1$,  and in \cite{DPSTT2} for the full range of  $1\leq p<\infty$. 

\begin{Theorem}\textnormal{\cite{VT159,DPSTT2}}\label{XT1} Let $1\le q<\infty$. Suppose that a subspace {$X_N\subset L_q(\Omega,\mu)$} satisfies the condition
\be\label{I6}
\e_k(X^q_N,L_\infty) \le  B (N/k)^{1/q}, \quad 1\leq k\le N,
\ee
where $B\ge 1$.
Then for a large enough constant $C(q)$ there exist
$m$ points
$\xi^1,\cdots, \xi^m \in  \Omega$ with
$$
m \le C(q)NB^{q}(\log_2(2BN))^2
$$
  such that for any $f\in X_N$
we have
$$
\frac{1}{2}\|f\|_q^q \le \frac{1}{m}\sum_{j=1}^m |f(\xi^j)|^q \le \frac{3}{2}\|f\|_q^q.
$$
\end{Theorem}

As we explained above the problem of universal discretization of the collection $\{X(n)\}_{n=1}^k$ is equivalent to the sampling discretization of the union $\cup_{n=1}^k X(n)$ of the corresponding subsets. Therefore, instead of bounds of the entropy numbers of the unit ball 
$X_N^q$ we are interested in the entropy bounds of the "unit ball"
$$
\Sigma_v^q(\D_N):= \{f\in \Sigma_v(\D_N)\,:\, \|f\|_q\le 1\}
$$
which is the union of the corresponding unit balls. 

The following version of Theorem \ref{XT1} 
 follows directly from its proof. 
 
 \begin{Theorem}\label{XT2} Let $1\le q<\infty$ and $1\leq v\leq N$. Suppose that a dictionary $\D_N$ is such that the set $\Sigma_v^q(\D_N)$ satisfies the condition
\be\label{Z1}
\e_k(\Sigma_v^q(\D_N),L_\infty) \le  B_1 (v/k)^{1/q}, \quad 1\leq k <\infty,
\ee
where $B_1\ge 1$.	Assume in addition that there exists a constant $B_2\ge 1$ such that 
\begin{equation}\label{2-3b} 
\|f\|_\infty \leq B_2 v^{1/q} \|f\|_q,\   \   \forall f\in \Sigma_v^q(\D_N).
\end{equation}
Then for a large enough constant $C(q)$ there exist
$m$ points
$\xi^1,\cdots, \xi^m\in  \Omega$ with
$$
m \le C(q)B_1^{q} v(\log(2B_2v))^2
$$
  such that for any $f\in  \Sigma_v(\D_N)$
we have
$$
\frac{3}{4}\|f\|_q^q \le \frac{1}{m}\sum_{j=1}^m |f(\xi^j)|^q \le \frac{5}{4}\|f\|_q^q.
$$
\end{Theorem}

Theorem \ref{XT2} also follows from  a  more general conditional theorem that will be proved in Section \ref{sec:5} (see Corollary \ref{Cor5.1}).

\begin{Remark}
	We point out that \eqref{Z1} implies   
	\be\label{2.3a} \|f\|_\infty \leq 3 B_1 v^{1/q},\   \  \forall f\in \Sigma^q_v(\D_N).
	\ee
Therefore, assumption (\ref{2-3b}) can be dropped with $B_2$ replaced by $3B_1$ in the bound on $m$.	 However, in  applications, the constant $B_2$ in \eqref{2-3b} may be significantly smaller than  $3B_1$.  For example, if $\D_N$ is a uniformly bounded orthonormal system with $\max_{f\in\D_N}\|f\|_\infty=1$, then we can take $B_2=1$.  
\end{Remark}

{\bf Proof of (\ref{2.3a}).} For $\Lambda\subset [1,N]\cap \N$ denote $X(\Lambda):= \sp(g_i)_{i\in\Lambda}$ and $X(\Lambda)^q := \{f\in X(\Lambda):\, \|f\|_q\le 1\}$. Clearly, (\ref{Z1}) implies the same bound for each $X(\Lambda)^q$ with $|\Lambda|=v$. Thus, it is sufficient to prove (\ref{2.3a}) for a $v$-dimensional subspace $X_v$. With a slightly worse constant $4B_1$ instead of $3B_1$ it  was proved in \cite[Remark 1.1]{DPSTT2}. We now show how to get a better constant.	Setting  $\va_1:= e_1(X_v^q,L_\infty)$, we can find two functions  $f_1, f_2 \in X_v^q$ such that $X_v^q \subset B_{L_\infty} (f_1, \va_1) \cup B_{L_\infty} (f_2, \va_1)$. 
	Since  $0\in X_v^q$, $0$ is contained in one of the two balls. Without loss of generality, we may assume that  $0\in B_{L_\infty} (f_1, \va_1)$ so that  $\|f_1\|_\infty \leq \va_1$.  Since $-f_2\in X_v^q$, we have either $-f_2 \in B_{L_\infty} (f_1, \va_1)$ or  $-f_2 \in B_{L_\infty} (f_2, \va_1)$, which implies   $\|f_2\|_\infty \leq 2\va_1$.  It then follows that $\|f\|_\infty\leq 3 \va_1$ for all $f\in X_v^q$. 
	This together with \eqref{Z1} proves 
	\eqref{2.3a}.

Theorem \ref{XT2}    motivates us to estimate the characteristics $\e_k(\Sigma_v^q(\D_N),L_\infty)$. We now recall
some known general results, which turn out to be useful for that purpose. Let $\D_N=\{g_j\}_{j=1}^N$ be a system of  elements of cardinality $|\D_N|=N$ in a Banach space $X$. Consider the best $m$-term approximations of $f$ with respect to $\D_N$
$$
\sigma_m(f,\D_N)_X:= \inf_{\{c_j\};\Lambda:|\Lambda|=m}\|f-\sum_{j\in \Lambda}c_jg_j\|.
$$
For a set $W\subset X$, we define
$$
\sigma_m(W,\D_N)_X:=\sup_{f\in W}\sigma_m(f,\D_N)_X,\   \  m=1,2,\cdots, $$
and $  \sigma_0(W,\D_N)_X=\sup_{f\in W}\|f\|_X$.
The following Theorem \ref{XT3} was proved in \cite{VT138} (see also \cite{VTbookMA}, p.331, Theorem 7.4.3).
\begin{Theorem}\label{XT3} Let a compact $W\subset X$ be such that there exist a  system $\D_N\subset X$ with  $|\D_N|=N$, and  a number $r>0$ such that 
$$
  \sigma_m(W,\D_N)_X \le (m+1)^{-r},\quad m=0,1,\cdots, N.
$$
Then for $k\le N$
\begin{equation}\label{X3}
\e_k(W,X) \le C(r) \left(\frac{\log(2N/k)}{k}\right)^r.
\end{equation}
\end{Theorem}

For a given set $\D_N=\{g_j\}_{j=1}^N$ of elements we introduce the octahedron (generalized octahedron)
\be\label{X4}
A_1(\D_N) := \left\{f\,:\, f=\sum_{j=1}^N c_jg_j,\quad \sum_{j=1}^N |c_j|\le 1\right\}
\ee
and the norm $\|\cdot\|_A$ on $X_N$
$$
\|f\|_A := \inf\left\{ \sum_{j=1}^N |c_j|\,:\, f=\sum_{j=1}^N c_jg_j \right\}.
$$

We now use a known general result for a smooth Banach space. For a Banach space $X$ we define the modulus of smoothness
$$
\rho(u):= \rho(X,u) := \sup_{\|x\|=\|y\|=1}\left(\frac{1}{2}(\|x+uy\|+\|x-uy\|)-1\right).
$$
The uniformly smooth Banach space is the one with the property
$$
\lim_{u\to 0}\rho(u)/u =0.
$$
In this paper we only consider uniformly smooth Banach spaces with power type modulii of smoothness $\rho(u) \le \ga u^s$, $1< s\le 2$. The following bound is a corollary of greedy approximation results (see, for instance \cite{VTbookMA}, p.455).

 \begin{Theorem}\label{XT4} Let $X$ be $s$-smooth: $\rho(X,u) \le \gamma u^s$, $1<s\le 2$. 
 Then for any  normalized system $\D_N$ of cardinality $|\D_N|=N$ we have
 $$
 \sigma_m(A_1(\D_N), X) \le C(s)\gamma^{1/s}m^{1/s-1}.
 $$
 \end{Theorem}
 
 Note that it is known that in the case $X=L_p$ we have
 \be\label{X5-0}
 \rho(L_p,u) \le (p-1)u^2/2,\quad 2\le p<\infty.
 \ee
 
 We now proceed to a special case when $X=L_p$ and $\D_N=\Phi_N:=\{\ff_j\}_{j=1}^N$ is a uniformly bounded Riesz basis of $X_N:=[\Phi_N]:=\sp(\ff_1,\dots,\ff_N)$. Namely, we assume 
 that $\|\ff_j\|_\infty \le 1$, $1\le j\le N$ and for any $(a_1,\cdots, a_N) \in\R^N,$
\begin{equation}\label{Riesz}
R_1 \left( \sum_{j=1}^N |a_j|^2\right)^{1/2} \le \left\|\sum_{j=1}^N a_j\ff_j\right\|_2 \le R_2 \left( \sum_{j=1}^N |a_j|^2\right)^{1/2},
\end{equation}
where $0< R_1 \le R_2 <\infty$. Assume in addition that for any $f\in X_N$ we have
\be\label{X5a}
\|f\|_\infty \le C_0\|f\|_{\log N}.
\ee

\begin{Theorem}\label{XT5}   
Assume that $\Phi_N$ is a uniformly bounded Riesz basis of $X_N:=[\Phi_N]$ satisfying (\ref{X5a}). Then   we have
 \be\label{X5}
 \e_k(\Sigma_v^2(\Phi_N),L_\infty) \le C(R_1,C_0) (\log N) \Bigl(\frac vk\Bigr)^{1/2},\quad   k=1,2,\dots.
 \ee
 \end{Theorem}
 \begin{proof} First of all, for any $f=\sum_{j\in G }a_j\ff_j$, $|G|=v$ we get
  \be\label{X6}
 \|f\|_A \le \sum_{j\in G} |a_j| \le v^{1/2} \left( \sum_{j\in G} |a_j|^2\right)^{1/2} \le R_1^{-1}v^{1/2}\|f\|_2.
 \ee
 Therefore, 
 $$
 \Sigma_v^2(\Phi_N) \subset R_1^{-1}v^{1/2}\Sigma_v^A(\Phi_N),$$
 where $$ R\Sigma_v^A(\Phi_N):=\{f\in\Sigma_v(\Phi_N)\,:\, \|f\|_A \le R\}.
 $$
 By Theorem \ref{XT4} with $s=2$ and by (\ref{X5-0}), we have that 
 for $p\in[2,\infty)$
 \be\label{X7-0}
 \sigma_m(\Sigma_v^2(\Phi_N),L_p) \le C R_1^{-1} v^{1/2} \sqrt p m^{-\frac12},\quad   m=1,2,\cdots,N.
 \ee
 Thus,   Theorem \ref{XT3} implies that  for $p\in[2,\infty)$
 \be\label{X7}
  \e_k(\Sigma_v^2(\Phi_N),L_p) \le C(R_1) (p \log(2N/k))^{1/2}(v/k)^{1/2},\quad   k=1,2,\dots,N.
 \ee
 Second, by (\ref{X5a})  we obtain  
 \be\label{X8}
   \e_k(\Sigma_v^2(\Phi_N),L_\infty)\le C_0\e_k(\Sigma_v^2(\Phi_N),L_{\log N}). 
 \ee 
 Combining (\ref{X7}) and (\ref{X8}) we get   
 \be\label{X9}
  \e_k(\Sigma_v^2(\Phi_N),L_\infty) \le C(R_1,C_0) (\log N)  (v/k)^{1/2},\quad   k=1,2,\dots,N.
 \ee

 Finally, for $k>N$ we use the inequalities
 $$
 \e_k(W,L_\infty) \le \e_N(W,L_\infty)\e_{k-N}(X_N^\infty,L_\infty)
 $$
 and 
 $$
 \e_{n}(X_N^\infty,L_\infty) \le 3(2^{-n/N}),\qquad 2^{-x} \le 1/x,\quad x\ge 1,
 $$
 to obtain (\ref{X9}) for all $k$.
 This completes the proof. 
 
 \end{proof}

\section{A step from $p=2$ to $1\le p<2$}
\label{Y}

In this section we show how Theorem \ref{XT5} proved in Section \ref{X} for $p=2$ can be extended to the case $1\le p<2$. This extension step is based on a general inequality for 
the entropy. For convenience, we  set $\Sigma_v(\D_N)=X_N:=[\D_N]$ for $v> N$.  

\begin{Lemma}\label{YL1} For $v=1,2,\dots, N$,   $1\le p< 2<q\leq \infty$ and 	$\ta:=(\frac 12-\frac 1q)/(\frac 1p-\frac1q)$, we have for $\va>0$
	\begin{equation}\label{0-9}
	\cH_{\va} (\Sigma_v^p(\D_N); L_{q}) \le \sum_{s=0}^\infty \cH_{ 2^{-3}a^{s-1} \va^{\ta} } (\Sigma_{2v}^2(\D_N); L_{q})+\cH_{\va^\ta} (\Sigma_{2v}^2(\D_N); L_{q}), 
	\end{equation}
	where $a=a(\ta) =2^{\frac \ta {1-\ta}}$.
	
\end{Lemma}

\begin{proof} In  the case when $v=N$ Lemma \ref{YL1} was proved in \cite{DPSTT2}.  A slight modification of the proof there works equally well for a general case of $1\leq v\leq N$. For completeness, we include the proof of this lemma here. First, we note that for any $\va_1, \va_2>0$, 
	\begin{equation}\label{2-2}
	\cH_{\va_1\va_2} (\Sigma_v^p(\D_N); L_q) \leq \cH_{\va_1} (\Sigma_v^p(\D_N); L_2)+\cH_{\va_2}(\Sigma_{2v}^2(\D_N); L_q).
	\end{equation}
	To see this,  let 
	$x_1,\cdots, x_{N_1}\in\Sigma_v^p(\D_N)$ and $y_1,\cdots, y_{N_2}\in \Sigma_{2v}^2(\D_N)$  be such that 
	\[\Sigma_v^p(\D_N) \subset \bigcup_{i=1}^{N_1} \left( x_i + \va_1 B_{L_2}\right)\    \   \text{and}\   \  \Sigma_{2v}^2(\D_N) \subset \bigcup_{j=1}^{N_2} (y_j  + \va_2 B_{L_q}),\]
	where  $N_1=N_{\va_1}(\Sigma_v^p(\D_N), L_2)$ and $N_2=N_{\va_2}(\Sigma_{2v}^2(\D_N), L_q)$.
	Since  $\Sigma_v(\D_N) +\Sigma_v(\D_N) \subset \Sigma_{2v}(\D_N)$, we have 
	\begin{align*}
	\Sigma_v^p(\D_N) &\subset \bigcup_{i=1}^{N_1} \left( x_i + \va_1 B_{L_2}\right)\cap \Sigma_v(\D_N) \subset  \bigcup_{i=1}^{N_1} \left( x_i + \va_1 \Sigma_{2v}^2(\D_N)\right)\\
	& \subset  \bigcup_{i=1}^{N_1} \bigcup_{j=1}^{N_2} \left( x_i + \va_1y_j +\va_1\va_2 B_{L_q}\right).
	\end{align*}
Inequality	 \eqref{2-2} then follows.

	Next,  setting
	$\va_1:=\va^{1-\ta}$ and $\va_2=\va^\ta$ in \eqref{2-2},  we reduce to showing that
	\begin{equation}\label{4-16}
	\cH_{\va_1} (\Sigma_v^p(\D_N); L_2) \leq \sum_{s=0}^\infty \cH_{ 2^{-3}a^{s-1} \va^{\ta} } (\Sigma_{2v}^2(\D_N); L_{q}).
	\end{equation}
	 	It will be shown that for $s=0,1,\ldots,$
	\begin{align}\label{3-19-a}
	\cH_{2^{s} \va_1} (\Sigma_v^p(\D_N); L_2) -\cH_{2^{s+1} \va_1} (\Sigma_v^p(\D_N); L_2)\leq \cH_{ 2^{-3}a^{s-1} \va^{\ta} } (\Sigma_{2v}^2(\D_N); L_{q})  ,
	\end{align}
	from which \eqref{4-16} will follow by taking  the sum over $s=0,1,\ldots$

	To show \eqref{3-19-a}, for  each nonnegative  integer $s$, let $\FF_s\subset \Sigma_v^p(\D_N)$ be a maximal $2^s \va_1$-separated subset of $\Sigma_v^p(\D_N)$ in the metric $L_2$; that is, $\|f-g\|_2\ge 2^s\va_1$ for any two distinct functions $f, g\in \FF_s$, and $\Sigma_v^p(\D_N)\subset \bigcup_{f\in\FF_s} B_{L_2} (f, 2^s\va_1)$.  Then
	\begin{equation}\label{4-17}
	\cH_{2^s \va_1}(\Sigma_v^p(\D_N); L_2) \leq  \log_2 |\FF_s| \leq \cH_{2^{s-1}\va_1}(\Sigma_v^p(\D_N); L_2).
	\end{equation}
	Let   $f_s\in \FF_{s+2}$  be such that
	$$\left| B_{L_2}(f_s, 2^{s+2} \va_1)\cap \FF_s\right|=\max_{f\in\FF_{s+2}} \left| B_{L_2}(f, 2^{s+2} \va_1)\cap \FF_s\right|.$$
	Since
	\begin{align*}
	\FF_s =\bigcup_{f\in \FF_{s+2}}\left(B_{L_2}(f, 2^{s+2} \va_1)\cap \FF_s\right)\subset \Sigma_v^p(\D_N),
	\end{align*}
	it follows that
	\begin{align}\label{4-18}
	|\FF_s|\leq |\FF_{s+2}| \left| B_{L_2}(f_s, 2^{s+2} \va_1)\cap \FF_s\right|.
	\end{align}
	Set
	$$ \cA_s:=\left\{ \frac{f-f_s}{2^{s+2} \va_1}:\   \        f\in B_{L_2} (f_s, 2^{s+2} \va_1) \cap \FF_s\right\}\subset \Sigma_{2v}(\D_N).$$
	Clearly,  for any $g\in\cA_s$,
	\begin{equation}\label{3-22-a}
	\|g\|_2\leq 1,\   \  \|g\|_p\leq  (2^{s+1} \va_1)^{-1}.
	\end{equation}
	On the one hand,  using  \eqref{4-17} and \eqref{4-18}, we obtain  that
	\begin{align}\label{3-23-a}
	\log_2|\cA_s|&\ge \log_2|\FF_s|-\log_2|\FF_{s+2}|\notag\\
	& \ge  \cH_{2^{s} \va_1} (\Sigma_v^p(\D_N); L_2) -\cH_{2^{s+1} \va_1} (\Sigma_v^p(\D_N); L_2) .
	\end{align}
	On the other hand, since $\frac 12 =\frac \ta p+\frac {1-\ta}q$,  using \eqref{3-22-a} and the fact that $\FF_s$ is $2^s\va_1$-separated in the $L_2$-metric, we have that  for any two distinct  $g',  g\in \cA_s$,
	\begin{align*}
	2^{-2} \leq  \|g'-g\|_2\leq \|g'-g\|_p^\ta \|g-g'\|_q^{1-\ta}\leq  \left (2^{s+1} \va_1\right)^{-\ta} \|g-g'\|_q^{1-\ta},
	\end{align*}
	which implies that
	\[ \|g'-g\|_q \ge 2^{-2}(2^{s-1} \va_1)^{\frac \ta {1-\ta}}=2^{-2} a^{s-1} \va^\ta.\]
	This together with  \eqref{3-22-a} means  that   $\cA_s$ is a $2^{-2} a^{s-1} \va^\ta $-separated subset of $\Sigma_{2v}^2(\D_N)$ in the metric $L_{q}$. We obtain
	\begin{align}\label{3-24-a}
	\log_2 |\cA_s| \leq \cH_{ 2^{-3} a^{s-1} \va^\ta  } (\Sigma_{2v}^2(\D_N); L_q).
	\end{align}
	Thus, combining  \eqref{3-24-a} with \eqref{3-23-a}, we prove  inequality \eqref{3-19-a}. 
\end{proof}

Lemma \ref{YL1} with $1\le p<2$, $q=\infty$, $\ta=p/2$ and Theorem \ref{XT5} imply the following 
bound for the entropy numbers.

\begin{Theorem}\label{YT1}   
Assume that $\Phi_N$ is a uniformly bounded Riesz basis of $X_N:=[\Phi_N]$ satisfying (\ref{X5a}). Then  for $1\leq p\leq 2$,  we have
 \be\label{X5}
 \e_k(\Sigma_v^p(\Phi_N),L_\infty) \le C(p,R_1,C_0) (\log N)^{2/p} (v/k)^{1/p},\quad   k=1,2,\dots.
 \ee
 \end{Theorem}

\section{Proof of Theorem \ref{IT3}}\label{Z}
 
   Theorem \ref{YT1}  provides bounds on the entropy numbers $ \e_k(\Sigma_v^p(\Phi_N),L_\infty)$ under additional assumption \eqref{X5a}. Thus, a combination of Theorem \ref{YT1} with Theorem \ref{XT2} implies the statement of Theorem \ref{IT3} under extra assumption (\ref{X5a}).    However,   \eqref{X5a} is not assumed in Theorem ~\ref{IT3}.  Below we  give a proof of Theorem ~\ref{IT3}.

 We need  the following lemma proved in \cite{DPSTT2}.

\begin{Lemma}\label{lem-4-1} 	\textnormal{\cite[Lemma 4.3] {DPSTT2}}  Let $1\leq p<\infty$  be a fixed number. Assume that   $X_N$ is  an $N$-dimensional subspace of $L_\infty(\Omega)$ satisfying  the following condition:  for  some  parameter $\beta >0$ and constant $K\ge 2$
	\begin{equation}\label{6-4}
	\|f\|_\infty \leq (KN)^{\frac \beta p} \|f\|_p,\   \   \forall f\in X_N.
	\end{equation}
	Let $\{\xi_j\}_{j=1}^\infty$ be a sequence of independent random points  selected uniformly from the probability space $(\Omega,\mu)$.
	Then   there exists a  positive   constant  $C_\beta$  depending only on $\beta$  such that for any   $0< \va\leq \frac 12$ and  any integer   \begin{equation}
	m\ge C_\beta K^\beta \va^{-2}( \log\frac 2\va) N^{\beta+1}\log N,\label{4-3-a}
	\end{equation}  the   inequality
	\begin{align}
	(1-\va) \|f\|_p^p\leq  \frac 1m \sum_{j=1}^m| f(\xi_j)|^p \leq (1+\va)\|f\|_p^p,
	\end{align}holds  with probability
	$ \ge 1-m^{-N/\log K}$.
	
\end{Lemma}

For a set  $\Omega_m:=\{x_1,\cdots, x_m\}\subset \Omega$ and a function $f:\Omega_m\to \R$, we define $\|f\|_{L_\infty(\Omega_m)} :=\max_{1\leq j\leq m} |f(x_j)|$ and 
\[ \|f\|_{L_q(\Omega_m)}:=\Bigl(\frac 1m \sum_{j=1}^m |f(x_j)|^q\Bigr)^{\frac1q}\   \ \text{for $q<\infty$} .\]

Now we turn to the proof of   Theorem \ref{IT3}. Recall that  we do not assume   \eqref{X5a}.  First, since  the Riesz basis $\Phi_N:=\{\varphi_j\}_{j=1}^{N}$ is uniformly bounded by $1$ on $\Omega$, we have
by (\ref{Riesz}) 
	\[\|f\|_\infty \leq  R_1^{-1} N^{\frac12} \|f\|_2,\   \ \forall f\in X_N,\]
	which in turn implies that 
	\[\|f\|_\infty \leq  C(R_1) N^{\frac1p} \|f\|_p,\   \ \forall f\in X_N,\   \ 1\leq p\leq 2.\]
	Thus,  by Lemma \ref{lem-4-1} with $\beta=1$,
	there exists 
	a discrete set $\Omega_{m_1}:=\{\xi^1,\cdots, \xi^{m_1}\}\subset \Omega$  with 
	$$ 	   C^{-1}  N^{2}\log N\leq m_1\leq  C  N^{2}\log N $$
such that  for all $f\in X_N$, 
	\begin{align}\label{4-4}
	\frac 45 \|f\|_p^p\leq \|f\|_{L_p(\Omega_{m_1})}^p \leq \frac 65\|f\|_p^p  \  \ \text{and}\   \  	\frac 45 \|f\|_2^2\leq \|f\|_{L_2(\Omega_{m_1})}^2 \leq \frac 65\|f\|_2^2,
	\end{align}
where $C>1$ is an absolute constant. 	

Second,  we consider the discrete  norm $\|\cdot\|_{L_p(\Omega_{m_1})}$ instead of the norm $\|\cdot\|_{L_p(\Omega)}$.   By \eqref{4-4},   $\Phi_N$ is a uniformly bounded Riesz basis of the space $(X_N, \|\cdot\|_{L_2(\Omega_{m_1})})$, and moreover,
$$ \|f\|_{L_\infty(\Omega_{m_1})} \leq C(p, R_1, R_2)  v^{1/p} \|f\|_{L_p(\Omega_{m_1})},\   \  \forall f\in\Sigma_v(\Phi_N).$$
Since $\log m_1 \sim \log N$,   by the regular Nikolskii inequality for the norms $\ell_q^{m_1}$, $1\leq q\leq \infty$, we also have  
\[ \|f\|_{L_\infty(\Omega_{m_1})} \leq C \|f\|_{L_{\log N}(\Omega_{m_1})},\   \   \forall f\in X_N,\]
where $C>1$ is an absolute constant. 
Thus, by Theorem \ref{XT2} and  Theorem~\ref{YT1} applied to the discrete norm   $ \|\cdot\|_{L_p(\Omega_{m_1})}$,   we can find a subset $\Omega_m\subset \Omega_{m_1}$ with 
\[ m=|\Og_m| \leq C(p, R_1, R_2) v (\log N)^2 (\log (2v))^2\]
such that for any $f\in\Sigma_v(\Phi_N)$, 
\begin{equation}\label{4.5}
\frac 34 \|f\|_{L_p(\Og_{m_1})}^p \leq \|f\|_{L_p(\Og_{m})}^p \leq \frac 54 \|f\|_{L_p,(\Og_{m_1})}^p.
\end{equation}
Combining \eqref{4-4} with \eqref{4.5}, we obtain the stated result of Theorem \ref{IT3}.

	\section{A refined  version of the conditional theorem  }\label{sec:5}

	Let us first recall some notations. 
	  Let $(\Og, \mu)$ be a probability space.   For $1\leq p\leq \infty$, denote by  $L_p(\Og)$   the usual Lebesgue space $L_p$ defined with respect to the measure $\mu$ on $\Og$, and by  $\|\cdot\|_p$ the norm of $L_p(\Og)$. We also set  $$B_{L_p}:=\{f\in L_p(\Og):\  \ \|f\|_p\leq 1\},\   \ 1\leq p\leq \infty. $$

	In this section, we prove a refined version of the  conditional theorem for  sampling  discretization of  all integral norms $L_q$ of functions from a  more general    subset  $\CW\subset L_\infty$, which allows us to estimate  the number of  points needed for  the sampling discretization in terms of  an integral of the $\va$-entropy  $\mathcal{H}_\va (\CW, L_\infty)$, $\va>0$.

	\begin{Theorem}\label{thm-4-2}
		Let $1\leq p<\infty$, and let  $\CW$  be a set of uniformly bounded functions on $\Og$ with  $$ 
	1\leq 	R:=\sup_{f\in \CW}\sup_{x\in\Og}  |f(x)|<\infty.$$   Assume that  $\cH_{ t}(\CW,L_\infty)<\infty$ for every $t>0$, and 
		\begin{equation}\label{5.1b}
		 (\lambda\cdot \CW)\cap B_{L_p} \subset \CW\subset B_{L_p},\   \   \ \forall \lambda>0.
		\end{equation} 
			Then there exist positive constants $C_p, c_p$ depending only on $p$   such that for any  $\va\in (0, 1)$ and  any 
		integer 
		\begin{equation}\label{5-2-c}
		m\ge  C_p 	  \va^{-5 }\left(\int_{10^{-1}\va^{1/p}} ^{R} u^{\frac  p2-1}    \Bigl(\int_{  u}^{ R }\frac {\cH_{c_p \va t}(\CW,L_\infty)}t\, dt\Bigr)^{\frac 12} du\right)^2, 
		\end{equation}
		there exist  $m$ points $x_1,\cdots, x_m\in \Og$  such that for all $f\in  \CW$, 
		\begin{equation}\label{5.3b}
		(1-\va) \|f\|_p^p \leq \frac  1m \sum_{j=1}^m |f(x_j)|^p\leq (1+\va) \|f\|_p^p.
		\end{equation} 
	\end{Theorem}

	 In particular, Theorem \ref{thm-4-2}   allows us to prove   refined versions of  Theorem \ref{XT2} and  Theorem \ref{IT3},
	 where   the constants in the Marcinkiewicz type discretization are    replaced by $1-\va$ and $1+\va$ for an arbitrarily given $\va\in (0, 1)$.

	First,  we have the following refined version  of  Theorem \ref{XT2}.
	
	\begin{Corollary}\label{Cor5.1} Let $1\le q<\infty$ and $1\leq v\leq N$. Suppose that a dictionary $\D_N$ is such that 
		\be\label{Z1-5.4}
		\e_k(\Sigma_v^q(\D_N),L_\infty) \le  B_1 (v/k)^{1/q}, \quad k=1,2,\cdots,
		\ee
		where $B_1\ge 1$.
		Assume in addition that there exists a constant $B_2\ge 1$ such that 
		\begin{equation}\label{5-5b} 
		\|f\|_\infty \leq B_2 v^{1/q} \|f\|_q,\   \   \forall f\in \Sigma_v^q(\D_N).
		\end{equation}
		Then for  a large enough constant $C(q)$ and any $\va\in (0, 1)$,  there exist 
		$m$ points
		$\xi^1,\cdots, \xi^m\in  \Omega$ with
		$$
		m \le C(q)   \va^{-5-q} v B_1^{q}(\log (B_2v/\va))^2
		$$
		such that for any $f\in  \Sigma_v(\D_N)$
		$$
		(1-\va)\|f\|_q^q \le \frac{1}{m}\sum_{j=1}^m |f(\xi^j)|^q \le (1+\va) \|f\|_q^q.
		$$
	\end{Corollary}

\begin{proof} [Proof of Corollary \ref{Cor5.1}]
 	We  apply Theorem \ref{thm-4-2}  to  $p=q$, $\CW:=\Sigma_v^q(\D_N)$
 and $R=B_2 v^{1/q}$.	It is clear that  $\CW$ satisfies \eqref{5.1b} with $p=q$. Furthermore,  \eqref{Z1-5.4} implies that (see by \cite[Lemma  2.1]{DPSTT2}) 
	\begin{equation}\label{5.4}
	\cH_{t} ( \CW, L_\infty) \leq C(q)  v\cdot  (B_1/t)^q,\    \  t>0.
	\end{equation}
Finally, a straightforward calculation using \eqref{5.4} then shows that 
\begin{align*}
&\va^{-5}  \left(\int_{10^{-1}\va^{1/q}} ^{B_2 v^{1/q}} u^{\frac  q2-1}    \Bigl(\int_{  u}^{ B_2 v^{1/q} }\frac{\cH_{c_q \va t}(\CW,L_\infty)} t \, dt\Bigr)^{\frac 12} du\right)^2\\
&\leq   C (q) \va^{-5-q}  B_1^q v (\log (B_2 v/\va))^2.
\end{align*}	
 Corollary \ref{Cor5.1} then follows from  Theorem \ref{thm-4-2}. 	 
\end{proof}

Using Corollary \ref{Cor5.1} and following the proof in Section \ref{Z}, we can also obtain the    $\va$-version of Theorem \ref{IT3}. 	
	
	 \begin{Corollary}\label{Cor5.2} Let  $\Phi_N$ be  a uniformly bounded Riesz basis of $X_N:=\mspan(\Phi_N)\subset L_2(\Og)$ satisfying \eqref{Riesz} for some constants $0<R_1\leq R_2$. 
	  Let $1\le p\le 2$, and let $1\leq v\leq N$ be an integer.  	Then for a large enough constant $C=C(p,R_1,R_2)$,  and any $\va\in (0, 1)$,  there exist 
		$m$ points
		$\xi^1,\cdots, \xi^m\in  \Omega$ with
		$$
		m \le C\va^{-p-5} v(\log N)^2(\log(2v\va^{-1}))^2
		$$
		such that for any $f\in  \Sigma_v(\Phi_N)$
		we have
		\begin{equation*}
		(1-\va)\|f\|_p^p \le \frac{1}{m}\sum_{j=1}^m |f(\xi^j)|^p \le (1+\va)\|f\|_p^p.
		\end{equation*}
	\end{Corollary}

	The rest of this section is devoted to the proof of Theorem \ref{thm-4-2}, which  is  close to the proof of Theorem~\ref{XT1} from Section \ref{X} given   in \cite{DPSTT1}. 
	We need  the following  lemma:

	\begin{Lemma}\label{BL2} \textnormal{\cite[Lemma 2.4]{DPSTT1}} Let $\{\cF_j\}_{j\in G}$ be a collection of finite sets of {bounded}
		functions from $L_1(\Omega,\mu)$. Assume that for each $j\in G$ and all $f\in \cF_j$ we have
		$$
		\|f\|_1 \le 1,\quad \|f\|_\infty:= \sup_{x\in\Omega}|f(x)| \le M_j.
		$$
		Suppose that positive numbers $\eta_j$ {$\in (0,1)$} and a natural number $m$ satisfy the condition
		$$
		2\sum_{j\in G} |\cF_j| \exp \left(-\frac{m\eta_j^2}{8M_j}\right) <1.
		$$
		Then there exists a set $\xi=\{\xi^\nu\}_{\nu=1}^m \subset \Omega$ such that for each $j\in G$ and for all $f\in \cF_j$ we have
		$$
		\left|\|f\|_1 - \frac{1}{m}\sum_{\nu=1}^m |f(\xi^\nu)|\right| \le \eta_j.
		$$
	\end{Lemma}

%
%
	
	\begin{proof} [Proof of Theorem \ref{thm-4-2}]
		Let 
		\[\CW_1:= \{  f/ {\|f\|_p}:\  \ f\in\CW\setminus\{0\}\}.\]
		Clearly, $\CW_1\subset \CW$, and it suffices to prove \eqref{5.3b} 
		 for all $f\in\CW_1$. 
		Let $c^\ast=c_p^\ast\in (0, \frac 12)$ be a  sufficiently small  constant depending only on $p$.
		Let $a:=c^\ast\va$. Let  $J, j_0$ be two integers  such that $j_0<0\leq J$, 
		\begin{equation}\label{5.8}
		(1+a)^{J-1}\le R < (1+a)^J\   \ \text{and}\   \ (1+a)^{j_0 p} \leq \frac  15 \va  \leq (1+a)^{(j_0+1) p }. 
		\end{equation}
		For $j\in\ZZ$, let 
		$$
		\cA_j := \cN_{a(1+a)^j}(\CW_1,L_\infty)\subset \CW_1
		$$
		denote the minimal $a(1+a)^j$-net of $\CW_1$ in the norm of $L_\infty$. For $j\in\ZZ$ and  $f\in\CW_1$, we  define  $A_j(f)$  to be  the function in  $ {\cA}_j$ that is  closest to $f$ in the $L_\infty$ norm. 
		Thus, $\|A_j (f)-f\|_\infty\leq a(1+a)^j$ for all $f\in\CW_1$ and $j\in\ZZ$. 
		
		Next, for $f\in \CW_1$ and $j> j_0$, define 
		$$
		U_j(f) := \{\bx\in\Og : |A_j(f)(\bx)| \ge (1+a)^{j-1}\},
		$$
		and 
		$$
		D_j(f) := U_j(f) \setminus \bigcup_{k\ge j+1} U_k(f).$$
		We also set 
		$$
		D_{j_0}(f) := \Omega \setminus \bigcup_{k>j_0} U_k(f).
		$$
		Note that by \eqref{5.8}, $U_j(f)=\emptyset$ for $j> J$. 
		Thus,  $\{D_{j}(f):\  \ j=j_0,\cdots, J\}$ forms a partition of the domain $\Og$. Define
		\begin{equation}\label{4-3-0}
		h(f,\bx) := \sum_{j=j_0+1}^{J} (1+a)^j \chi_{D_j(f)}(\bx),
		\end{equation}
		where $\chi_E(\bx)$ is a characteristic function of a set $E$.

		For  $\bx\in D_{j_0} (f)$, we have 
		\begin{align*}
		|f(\bx)|&\leq |A_{j_0+1} f(\bx)| + a(1+a)^{j_0+1} \leq (1+a)^{j_0} + a(1+a)^{j_0+1}\\
		&\leq (1+a)^{j_0} (1+2a),
		\end{align*}
		which in turn implies that 
		\[ |f(\bx)|^p \leq  (1+a)^{j_0p} (1+2a)^p\leq (1+a)^{j_0p} (1+C_p a) \leq \frac 54 (1+a)^{j_0 p}<\frac  14 \va. \]
		On the other hand, for $f\in D_j(f)$ and $j_0<j\leq J$, we have
		\begin{align*}
		|f(\bx)|&\ge |A_j f (\bx)| -a(1+a)^j  \ge (1+a)^j (1-2a),\   \ \text{and}\\
		|f(\bx)|&\leq |A_{j+1} f(\bx)|+a(1+a)^{j+1}  \leq (1+a)^j (1+2a),
		\end{align*}
		which  implies 
		\[\Bigl| |f(\bx)|^p -(1+a)^{jp}\Bigr| \leq C_p a \leq c^\ast C_p \va<\frac  \va 4.  \]
		Therefore, 
		\begin{equation}\label{4-5b}
		\Bigl| |f(\bx)|^p -|h(f, \bx)|^p \Bigr|\leq \frac  \va 4,\   \ \forall \bx\in\Og.
		\end{equation}

		For $j_0+1\leq j\leq J$,  let 
		$$
		\cF_j^p := \left\{ (1+a)^{pj}\chi_{D_j(f)}:\   f\in  \CW_1\right\}.
		$$
		
		
		Our aim is to find  $m$ points $\xi^1, \cdots, \xi^m\in\Og$  for each $m$ satisfying \eqref{5-2-c}
		 so that the following inequality holds for all $f\in\cF_j^p$ and $ j_0<j\leq  J$:
		\begin{align}
		&\Bigl| \frac  1m \sum_{k=1}^m f(\xi^k) -\int_{\Og} f(x) \, d\mu(x) \Bigr|\leq \va_j, \label{4-5}
		\end{align} 
		where $\{\va_j\}_{j=j_0+1}^{J} \subset (0, 1)$ satisfies $\sum_{j=j_0+1}^{J} \va_j \leq \va/4$. 
		Once \eqref{4-5} is proved, we obtain by \eqref{4-3-0} that 
		\begin{align}
		\Bigl| \frac  1 m \sum_{j=1}^m |h(f, \xi^j)|^p-\|h(f)\|_p^p\Bigr|  \leq \frac  \va 4.
		\end{align}
		which, applying \eqref{4-5b}, will prove the desired inequality \eqref{5.3b}.

		 To see this, we apply Lemma \ref{BL2}  for the  collection of the above sets $\cF_j^p$, and notice  that 
		for $j_0<j\leq J$, 
		$$
		\|(1+a)^{pj}\chi_{D_j(f)}\|_1 \le \|h(f)\|_p^p\leq  \|f\|_p^p +\frac  \va4\le 2
		$$
		and 
		$$
		\|(1+a)^{pj}\chi_{D_j(f)}\|_\infty \le (1+a)^{pj} =: M_j.
		$$
		Thus,  by Lemma \ref{BL2}, 
	  it suffices  to show that  for   each integer $m$ satisfying \eqref{5-2-c},  one can find  a sequence   $\{\va_j\}_{j_0<j\leq J}\subset (0, 1)$ such that  
		\begin{align}
		&\sum_{j_0<j\leq J} \va_j \leq \frac  \va4\label{4-8}\\
		&\sum_{j=j_0+1}^{J} |\cF_j^p| \exp \Bigl( - \frac  {m\va_j^2} {8M_j} \Bigr)<\frac 12.\label{4-9}
		\end{align}

		To this end, we need to estimate the cardinalities of the sets $\FF_j^p$.   By definition, for each  $j_0< j\leq  J$,  the set $D_j(f)$ is uniquely determined by the functions $A_k(f)\in\cA_k$, $j\leq k\leq J$. As a result, we have 
		$$
		|\cF_j^p| \le |\cA_j|\times\cdots\times|\cA_{J}|=:L_j,
		$$
		and 
		\begin{align}
		\log L_j&\leq  \sum_{k=j}^{J} \log |\cA_k|\leq \sum_{k=j}^{J} \cH_{a(1+a)^k}(\CW,L_\infty)\notag\\
		&\leq \frac  1 {\log (1+a)} \sum_{k=j}^{J} \int_{ a( 1+a)^{k-1}}^{a(1+a)^{k}}\cH_{t}(\CW,L_\infty)\frac  {dt}t\notag\\
		&\leq C\va^{-1}  \int_{ ( 1+a)^{j-1}}^{R }\cH_{at}(\CW,L_\infty)\frac  {dt}t.\label{4-10}
		\end{align}

		For each $j_0<j\leq J$, we choose $\va_j>0$   so that 
		\begin{align}
		&\log (\lambda L_j) = \frac  {m\va_j^2}{16 M_j},\   \    \  \text{that is}\   \  \va_j:=4\sqrt{M_j} \sqrt{\log (\lambda  L_j)} m^{-\frac 12},\label{4-11a}
		\end{align}
		where $\lambda>1$ is a large  absolute constant to be specified later. 
		Then 
		\begin{align*}
		\sum_{j=j_0+1}^{J}\va_j &=4m^{-1/2} \sum_{j=j_0+1}^{J} (M_j \log (\lambda L_j))^{\frac 12}\\
		&\leq 4 m^{-1/2} \sqrt{\log \lambda }\sum_{j=j_0+1}^{J} (M_j \log ( L_j))^{\frac 12} ,
		\end{align*}
		and hence \eqref{4-8} is ensured once 
		\begin{equation}\label{4-11}
		m\ge  \va^{-2} \Bigl(\sqrt{\log \lambda } \sum_{j=j_0+1}^{J} (M_j \log (L_j))^{\frac 12}\Bigr)^2.
		\end{equation}

		However, using \eqref{4-10}, we have
		\begin{align*}
		\sum_{j=j_0+1}^{J} (M_j \log L_j)^{\frac 12}&\leq C\va^{-\frac 12}\sum_{j=j_0+1}^{J}  (1+a)^{pj/2}  \Bigl(\int_{ ( 1+a)^{j-1}}^{R }\cH_{at}(\CW,L_\infty)\frac  {dt}t\Bigr)^{\frac 12}\\
		&\leq C_p \va^{-\frac 32}\sum_{j=j_0+1}^{J} \int_{(1+a)^{j-2}} ^{(1+a)^{j-1}} u^{\frac  p2-1}    \Bigl(\int_{  u}^{ R }\cH_{at}(\CW,L_\infty)\frac  {dt}t\Bigr)^{\frac 12} du \\
		&\leq C_p \va^{-\frac 32}\int_{10^{-1}\va^{1/p}} ^{R} u^{\frac  p2-1}    \Bigl(\int_{  u}^{ R }\cH_{c_p \va t}(\CW,L_\infty)\frac  {dt}t\Bigr)^{\frac 12} du.
		\end{align*}
		This combined with \eqref{4-11} implies that  \eqref{4-8} is ensured by  \eqref{5-2-c}.  
		
		Finally, we prove \eqref{4-9}. Indeed, using \eqref{4-11a}, we have
		\begin{align*}
		&\sum_{j=j_0+1}^{J} |\cF_j^p| \exp \Bigl( - \frac  {m\va_j^2} {8M_j} \Bigr)\leq 
		\lambda \sum_{j=j_0+1}^{J}  L_j \exp \Bigl( - \frac  {m\va_j^2} {8M_j} \Bigr)=\sum_{j=j_0+1}^{J}  \exp \Bigl(\log (\lambda L_j)  - \frac  {m\va_j^2} {8M_j} \Bigr)\\
		&=  \sum_{j=j_0+1}^{J}  \exp \Bigl( - \log (\lambda L_j) \Bigr)=\frac  1\lambda \sum_{j=j_0+1}^{J} \frac  1 {L_j}\leq \frac  1 {\lambda}  \sum_{j=j_0+1}^{J} \frac  1 {N_{a(1+a)^j}(\CW, L_\infty)},
		\end{align*}
		where the last step uses the fact that 
		\[L_j \ge |\cA_j| =N_{a(1+a)^j}(\CW, L_\infty).\]
		
		We   claim that 
		\begin{equation}\label{4-13}
		\cH_{t} (\CW, L_\infty) \ge \log \frac  {R} {4t},\   \   \   \forall   0<t<R.
		\end{equation} 
		To see this,  let $f^\ast\in \CW$ be such that $\|f^\ast\|_\infty =R$. Let $k=[\frac  R t]$. 
		Define $f_j=\frac { 2jt}R f^\ast$ for $0\leq j\leq k/2$. 
		Then $\{f_j\}_{0\leq j\leq \frac  k2} \subset \CW$ is  $2t$-separated in $L_\infty$-norm. It follows that 
		\[ N_{t} (\CW, L_\infty) \ge \frac  k2 \ge \frac  14 \frac  R t, \]
		which shows \eqref{4-13}.

		Now using \eqref{4-13}, we obtain 
		\begin{align*}
		\frac  1 {\lambda}  \sum_{j=j_0+1}^{J} \frac  1 {N_{a(1+a)^j}(\CW, L_\infty)}  \leq C R^{-1} \lambda^{-1}  \sum_{j=j_0+1}^{J} (1+a)^j\leq C \lambda^{-1}<1,
		\end{align*}
		provided that $\lambda>1$ is large enough. This proves \eqref{4-9}.

	\end{proof}

	\section{ Concluding remarks on  sampling discretization of   $L_p$ norms  for $2<p<\infty$}\label{sec:6}

	In this section, we give a few remarks on  sampling discretization of   $L_p$ norms  for $2<p<\infty$.
	
	\begin{enumerate}
		\item  The following  Nikolskii type inequality plays an important role in the  proof of Theorem \ref{IT3}:   
	\begin{equation}
	 \|f\|_\infty \leq C v^{\frac  1p}\|f\|_p,\   \    \  \forall f\in \Sigma_v(\Phi_N),\end{equation} 
	where the constant $C$ is independent of $f$, $v$ and $N$. 
 This inequality holds for $1\leq p\leq 2$  whenever  $\Phi_N$ is a uniformly bounded Riesz basis of $X_N$.  	 However, this is no longer  true for $p>2$.
	For example, take   $N=2^v$ and consider  the system  
	$$\Phi_N =\{ e^{2\pi i j x}\}_{j=1}^N$$ 
	on the interval $[0,1]$ equipped with the usual Lebesgue measure.  	
	 By the Littlewood-Paley inequality, we have that  for $f(x)=\sum_{j=1}^v e^{2\pi \mathbf i 2^j x}\in\Sigma_v(\Phi_N)$ and $2<p<\infty$, 
	\begin{align*}
	\|f\|_\infty =v > C  v^{\frac  1p}\|f\|_p\asymp v^{\frac  12+\frac  1p}.
	\end{align*}

\item Let  $\Phi_N$  be  a uniformly bounded Riesz basis of $X_N\subset L_2$ satisfying \eqref{X5a}.  By monotonicity of the $L_p$ norms, we have that for any integer  $1\leq v\leq N$, 
	\[  \Sigma^p_v(\Phi_N)\subset \Sigma^2_v(\Phi_N),\   \   p>2,\]
	which in particular implies that 
	\[  \sup_{f\in\Sigma_v^p(\Phi_N)} \|f\|_\infty \leq   \sup_{f\in\Sigma_v^2(\Phi_N)} \|f\|_\infty\leq  C v^{1/2}.    \]
Moreover,  using 	Theorem \ref{XT5}, we have that for $p>2$ and all integer $k\ge 1$,  
		\be\label{X5-6.2}
	\e_k(\Sigma_v^p(\Phi_N),L_\infty) \le 	\e_k(\Sigma_v^2(\Phi_N),L_\infty)\leq C\cdot  (\log N) \Bigl(\frac vk\Bigr)^{1/2},
	\ee	
 which    also yields  
	\begin{equation}\label{6.3}
	\cH_{t} ( \Sigma_v^p(\Phi_N), L_\infty) \leq C(p)  v\cdot  \Bigl(\frac {\log N} t\Bigr)^2,\    \    \  \forall  t>0.
	\end{equation}
On the other hand, a straightforward calculation  shows that  for any $\va\in (0, 1)$ and $p>2$, 
	\begin{align*}
	&\va^{-5}  \left(\int_{10^{-1}\va^{1/p}} ^{C v^{1/2}} u^{\frac  p2-1}    \Bigl(\int_{  u}^{ C v^{1/2} }\frac{\cH_{c_p \va t}( \Sigma_v^p(\Phi_N),L_\infty)} t \, dt\Bigr)^{\frac 12} du\right)^2\\
	&\leq   C (p) \va^{-7}   v^{p/2}  (\log N)^2.
	\end{align*}

Thus, an application of Theorem \ref{thm-4-2}   leads to

	\begin{Theorem}   Assume that $\Phi_N$ is a uniformly bounded Riesz basis of $X_N:=\mspan(\Phi_N)$ satisfying \eqref{Riesz} for some constants $0<R_1\leq R_2$.
		Let $2<p<\infty$ and let  $1\leq v\leq N$ be an integer. 		
		 Then for a large enough constant $C=C(p,R_1,R_2)$ and any $\va\in (0, 1)$,   there exist 
		$m$ points  $\xi^1,\cdots, \xi^m\in  \Og$  with 
				\begin{equation*}
		m\leq C\va^{-7}       v^{p/2} (\log N)^2,
		\end{equation*}
			such that for any $f\in  \Sigma_v(\Phi_N)$, 
		\[ (1-\va) \|f\|_p^p \leq \frac   1m \sum_{j=1}^m |f(\xi^j)|^p\leq (1+\va) \|f\|_p^p. \]	
	\end{Theorem}

	\end{enumerate}

\Addresses

\end{document}